\documentclass{amsart}[12pt]
\def\doctype{}

\usepackage{latexsym,amssymb,amsfonts,bm}
\usepackage{fancyhdr,caption}
\usepackage{youngtab}
\usepackage{graphicx}
\usepackage{tikz}
\usetikzlibrary{matrix}
\usepackage{mathrsfs}
\usepackage{hyperref}

\newcommand\lam{\lambda}

\newcommand{\comment}[1]{}

\numberwithin{equation}{section}

\fancypagestyle{titlepage}{
\fancyhead[R]{\doctype}
\fancyhead[CL]{}
\cfoot{\vspace{5pt} \thepage}
}

\let\oldsection\section
\newcommand\boldsection[1]{\oldsection{\bf #1}}
\newcommand\starsection[1]{\oldsection*{\bf #1}}
\makeatletter
\renewcommand\section{\@ifstar\starsection\boldsection}
\makeatother

\newtheoremstyle{theorem}
  {12pt}		  
  {0pt}  
  {\sl}  
  {\parindent}     
  {\bf}  
  {. }    
  { }    
  {}     
\theoremstyle{theorem}
\newtheorem{thm}{Theorem}[section]  
\newtheorem{lemma}[thm]{Lemma}     
\newtheorem{cor}[thm]{Corollary}

\newtheorem{prop}[thm]{Proposition}

\newtheoremstyle{definition}
  {12pt}		  
  {0pt}  
  {}  
  {\parindent}     
  {\bf}  
  {. }    
  { }    
  {}     
\theoremstyle{definition}

\newtheorem{ex}[thm]{Example}

\renewcommand{\proofname}{Proof}

\makeatletter
\renewenvironment{proof}[1][\proofname]{\par
  \pushQED{\qed}%
  \normalfont \partopsep=\z@skip \topsep=\z@skip
  \trivlist
  \item[\hskip\labelsep
        \scshape
    #1\@addpunct{.}]\ignorespaces
}{%
  \popQED\endtrivlist\@endpefalse
}
\makeatother

\makeatletter
\renewcommand*\@maketitle{%
  \normalfont\normalsize
  \@adminfootnotes
  \@mkboth{\@nx\shortauthors}{\@nx\shorttitle}%
  \global\topskip42\p@\relax 
  \@settitle
  \ifx\@empty\authors \else {\vskip 1em
\vtop{\centering\shortauthors\@@par}} \fi
  \ifx\@empty\@date \else {\vskip 1em \vtop{\centering\@date\@@par}}\fi 
  \ifx\@empty\@dedicatory
  \else
    \baselineskip18\p@
    \vtop{\centering{\footnotesize\itshape\@dedicatory\@@par}%
      \global\dimen@i\prevdepth}\prevdepth\dimen@i
  \fi
  \@setabstract
  \normalsize
  \if@titlepage
    \newpage
  \else
    \dimen@34\p@ \advance\dimen@-\baselineskip
    \vskip\dimen@\relax
  \fi
} 
\renewcommand*\@adminfootnotes{%
  \let\@makefnmark\relax  \let\@thefnmark\relax
  \ifx\@empty\@subjclass\else \@footnotetext{\@setsubjclass}\fi
  \ifx\@empty\@keywords\else \@footnotetext{\@setkeywords}\fi
  \ifx\@empty\thankses\else \@footnotetext{%
    \def\par{\let\par\@par}\@setthanks}%
  \fi
\thispagestyle{titlepage}
}
\makeatother

\title{On the Algebraic Combinatorics of Injections \\ and its Applications to Injection Codes}

\author{Peter J.~Dukes}
\address{\rm Peter J.~ Dukes:
Mathematics and Statistics,
University of Victoria, Victoria, BC
}
\email{dukes@uvic.ca}

\author{Ferdinand Ihringer}
\address{\rm Ferdinand Ihringer:
Department of Mathematics: Analysis, Logic and Discrete Mathematics,
Ghent University,
Ghent, Belgium
}

\email{Ferdinand.Ihringer@gmail.com}

\author{Nathan Lindzey}
\address{\rm Nathan Lindzey:
Computer Science,
University of Colorado Boulder,
Boulder, CO
}

\email{Nathan.Lindzey@colorado.edu}

\thanks{Research of P.J.~Dukes is supported by NSERC grant 312595--2017. F. Ihringer is supported by a postdoctoral fellowship of the Research Foundation - Flanders (FWO)}

\date{\today}

\begin{document}

\begin{abstract}
We consider the algebraic combinatorics of the set of injections from a $k$-element set to an $n$-element set. In particular, we give a new combinatorial formula for the spherical functions of the Gelfand pair $(S_k \times S_n, \text{diag}(S_k) \times S_{n-k})$.
We use this combinatorial formula to give new Delsarte linear programming bounds on the size of codes over injections.
\end{abstract}

\maketitle
\hrule


\section{Introduction}

Let $S_n$ denote the symmetric group on $n$ elements and let $S_{k,n}$ denote the set of injections (alternatively, partial permutations) $\sigma$ from $[k] := \{1,2,\cdots, k\}$ to $[n]$. Note that $S_{n,n}=S_n$, so this is a natural generalization of the symmetric group. The goals of this paper are two-fold.  Firstly, we investigate the algebraic combinatorics of injections. In particular, we investigate the \textit{injection scheme}~\cite{LR,M}, an association scheme naturally associated with injections which is a simultaneous generalization of the Johnson association scheme and the group association scheme of $S_n$.  Secondly, we apply this theory to analyze injections in the coding-theoretic sense. That is, we use the character table of the injection scheme to compute, for a wide range of parameters, upper bounds on the number of partial permutations with prescribed minimum Hamming distance (or, in general, allowed distance sets).

\subsection{The Algebraic Combinatorics of Injections}

Throughout this work, let $G_{k,n} := S_k \times S_n$ and $K_{k,n} := \text{diag}(S_k) \times S_{n-k}$. We shall investigate the algebraic combinatorics of $S_{k,n}$ via the Gelfand pair $(G_{k,n},K_{k,n})$. The spherical functions of $(G_{k,n},K_{k,n})$ have combinatorial significance, as they describe the eigenvalues of a natural family of graphs defined over $S_{k,n}$, i.e., the character table of the injection association scheme. We begin with a brief overview of previous work related to the subject.

Diaconis and Shahshahani~\cite{Diaconis} observed that $(G_{k,n},K_{k,n})$ is a Gelfand pair by showing the double coset algebra $\mathbb{C}[K_{k,n} \backslash G_{k,n} / K_{k,n} ]$ is commutative. Later, Greenhalgh~\cite{Greenhalgh} found a closed expression for the spherical functions of $(G_{k,n},K_{k,n})$ evaluated at the double coset $K_{k,n} \backslash (k,k+1) / K_{k,n}$, equivalently, the eigenvalues of the graph over $S_{k,n}$ where $\sigma,\sigma'$ is an edge if their respective mappings agree on all but one symbol of the domain. Using this expression, he showed that the mixing time of the uniform random walk on this graph is approximately $(n-k)\log n + cn$ for some constant $c > 0$.

In quantum computing, the algebraic combinatorics of $S_{k,n}$ has been used to show adversarial lower bounds on the time-complexity of the \textsc{Collision}, \textsc{Set-Equality}, and \textsc{Index-Erasure} problems. These lower bounds are derived from properties of the dual characters and Krein parameters of the injection scheme ~\cite{Ambainis,BR,LR}, and are expressed in terms of the spherical functions of $(G_{k,n},K_{k,n})$.

Greenhalgh~\cite{Greenhalgh} posed the question of investigating the spherical functions of $(G_{k,n},K_{k,n})$, as spherical functions often correspond to interesting families of orthogonal polynomials (e.g., special functions). For the case $k = n-1$, the so-called ``unbalanced" pair $(S_{n-1} \times S_n, \text{diag}(S_{n-1}))$, Strahov~\cite{Strahov} showed that many of the classical results in the theory of symmetric functions have unbalanced analogues. In particular, he gave a Murnaghan-Nakayama type rule and a Jacobi-Trudi identity for evaluating its spherical functions. Note that the ``balanced" pair $(S_{n} \times S_n, \text{diag}(S_{n}))$ recovers the classical representation theory of $S_n$ (see \cite{MacDonald95}).

Such expressions for the cases $2 \leq k \leq n-2$ are not known, and to what extent the classical representation theory of the symmetric group carries over to these cases is an intriguing question. Indeed, the absence of useful combinatorial formulas for the spherical functions of $(G_{k,n},K_{k,n})$ has been a major obstacle in each of the areas above. 

We make some progress in this direction by giving a combinatorial formula for the spherical functions of $(G_{k,n},K_{k,n})$. The formula is significantly more revealing than the known formulas, and it is much easier to compute. It can be used to estimate the eigenvalues and ranks of matrices in the Bose-Mesner algebra of the injection scheme, in special cases, giving exact closed-form expressions (we do not pursue this direction in this paper), and it also allows us to efficiently compute the character tables of injection schemes for explicit parameters $k$ and $n$.

\subsection{Injection Codes}

We now outline the coding-theoretic framework, starting first with the well-studied case of permutations.  The Hamming distance between two permutations $\sigma, \tau \in S_n$ 
is the number of non-fixed points of $\sigma \tau^{-1}$, or, equivalently, the number of disagreements when $\sigma$ and $\tau$ are written as words in single-line notation.  For example, $1234$ and $3241$ are at distance three. This notion naturally generalizes to injections \cite{Dukes}.

A \emph{permutation code} PC$(n,d)$ is a subset $\Gamma$ of $S_n$ such that the distance between 
any two distinct elements of $\Gamma$ is at least $d$.  The language of classical coding theory is often used: elements of $\Gamma$ are \emph{words}, $n$ is the \emph{length} of the code, and the parameter $d$ is the \emph{minimum distance}, although for our purposes it is not important whether distance $d$ is ever achieved.  Permutation codes are also called \emph{permutation arrays} by some authors, where the words are written as rows of an $n \times |\Gamma|$ array.

The investigation of permutation codes essentially began with the articles \cite{DV,FD}.
After some years of inactivity, permutation codes enjoyed a resurgence, due in part to their applications to error-correction over certain channels and in tandem with the ongoing development of computational discrete optimization.

For positive integers $n \ge d$, we let $M(n,d)$ denote the maximum size of a PC$(n,d)$.  It is easy to see that $M(n,1)=M(n,2)=n!$, and that $M(n,n)=n$.  The Singleton bound $M(n,d) \le n!/(d-1)!$ holds.  The alternating group $A_n$ shows that $M(n,3)=n!/2$.  More generally, a sharply $k$-transitive subgroup of $S_n$ furnishes a permutation code of (maximum possible) size $n!/(n-k)!$.  For instance, the Mathieu groups $M_{11}$ and $M_{12}$ are maximum PC$(11,7)$ and PC$(12,7)$, respectively.  On the other hand, determination of $M(n,d)$ absent any algebraic structure appears to be a difficult problem.  A table of bounds on $M(n,d)$ can be found in \cite{SM}. 

A relaxation known as an \emph{injection code} was introduced in \cite{Dukes}.  As the name suggests, one considers the problem of packing, with resepect to Hamming distance, injections (alternatively partial permutations), of a fixed length $k$ using the alphabet $[n]=\{1,\dots,n\}$.  Let $M(n,k,d)$ denote the maximum size of a family of such injections with pairwise Hamming distance at least $d$.  For example, we have $M(n,k,k)=n$ from (for the lower bound) the cyclic shifts of, say, $12\cdots k$ and (for the upper bound) the pigeonhole principle.  The problem of determining $M(n,k,d)$ is, like its counterpart for permutations, challenging in general.  However, as we illustrate in Section~\ref{sec:codes} to follow, there are various closely related problems in combinatorics motivating further study of injection codes.

\subsection{Outline}

The outline of the paper is as follows.
In Section 2, we review integer partitions and tableaux, and generalize the Robinson-Schensted-Knuth correspondence to the setting of injections.  Then, in Section 3, we set up the algebraic combinatorics for injections.  In particular, we introduce the injection scheme, its associated Gelfand pair, and give a formula for its spherical functions in terms of characters of symmetric groups.  The main result of Section 4 pushes this to a purely combinatorial description using a canonical basis for $\mathbb{C}[S_{k,n}]$ in terms of tableaux.  In Section 5, we state the Delsarte linear programming bound for injection codes in terms of characters for the injection scheme. Our computations for parameters $k \le n \le 15$ are reported as tables of new or improved bounds.  Additionally, we offer some motivation for these bounds by noting several combinatorial problems closely connected with injection codes.  We conclude with a few open problems that naturally follow our work.

\section{Tableaux and Injections}

We give a brief overview of some tableau terminology, see~\cite{Sagan} for a more detailed treatment. Let $\lambda = (\lambda_1,\lambda_2,\cdots,\lambda_\ell) \vdash n$ denote an (integer) \emph{partition} of $n \in \mathbb{N}$. Let $\ell(\lambda)$ denote the \emph{length} of $\lambda$, that is, the number of parts in the partition. To each $\lambda \vdash n$, we may associate a \emph{tableau} $t$, a left-justified array of cells with $\ell$ rows and $\lambda_i$ cells in $i$th row. Let $\lambda^\top$ denote the \emph{transpose} partition, that is, the partition obtained by interchanging the rows and columns of $\lambda$'s tableau. For any $\mu = (\mu_1,\mu_2,\cdots,\mu_k) \vdash m$ such that $\mu_i \leq \lambda_i$ for all $1 \leq i \leq k$, the \emph{skew-tableau} $\lambda / \mu$ is the array of cells obtained by removing the cells corresponding to $\mu$ from $\lambda$. A skew-tableau is a \emph{horizontal strip} if no two of its cells lie in the same column.

We say a tableau $t$ of shape $\lambda \vdash n$ is a \emph{Young tableau} if its cells are assigned a unique $i \in [n]$. A Young tableau is \emph{standard} if the cells are in ascending order from left to right in each row, and in ascending order from top to bottom in each column. Let $f^\lambda$ denote the number of standard Young tableau of shape $\lambda \vdash n$. A \emph{tabloid} $\{t\}$ is a Young tableau such that cells in each row are unordered. For any tabloid $\{t\}$ with $n$ cells, let $\text{row}_{\{t\}}(i)$ denote the index of the row of $\{t\}$ that contains $i \in [n]$, and let $\text{col}_{t}(i)$ denote the index of the column of $t$ that contains $i \in [n]$.

A well-known fact is that the symmetric group $S_n$ on $n$ symbols admits the following representation-theoretic count:
\begin{align}\label{eq:rs}
	|S_n| =  \sum_{\lambda \vdash n} \left(f^\lambda \right)^2.
\end{align}
 An elegant combinatorial proof of this fact follows from \emph{Robinson-Schensted Correspondence}, a well-known combinatorial procedure that associates to each permutation $\sigma \in S_n$ a unique pair of standard Young tableaux of the same shape, and vice versa (see~\cite{Sagan}).
 	
Knuth generalized this correspondence to a wider class of combinatorial objects called \emph{generalized permutations}, which are $2 \times m$ arrays of integers
$$\left(\begin{array}{cccc}
i_1 & i_2 & \cdots & i_m\\
j_1 & j_2 & \cdots & j_m\
\end{array}\right) \text{ such that } i_1 \leq \cdots \leq i_m \text{ and if } i_r = i_{r+1}, \text{then } j_r \leq j_{r+1}.$$
\emph{Robinson-Schensted-Knuth Correspondence} (RSK) associates a pair of semistandard Young tableau of the same shape to each generalized permutation, and vice versa (see~\cite{Sagan}).
We may encode an injection $1 \mapsto j_1,2 \mapsto j_2,\cdots,k \mapsto j_k =: (j_1, j_2,\cdots, j_k)$ as a generalized permutation:
$$\left(\begin{array}{cccccccc}
1 & 2 & \cdots & k & k+1 & \cdots & k+1\\
j_1 & j_2 & \dots & j_k & j_{k+1} & \cdots & j_n\\
\end{array}\right),$$
where $j_{k+1}, \cdots, j_n \in [n] \setminus \{j_1,\cdots,j_k\}$ are ordered from least to greatest.  Applying the RSK algorithm to the encoded injections associates to each injection a standard Young tableau $P$ and a semistandard Young tableau $Q$ of the same shape $\lambda \vdash n$. The subtableau of cells labeled $k+1$ in $Q$ form a horizontal strip on $n-k$ cells. Removing this horizontal strip results in a standard Young tableau of shape $\mu \vdash k$ such that $\lambda / \mu$ is a horizontal strip, and so we arrive at the following theorem.
\begin{thm}
RSK gives an explicit bijection between $S_{k,n}$ and pairs $(P,Q)$ where $P$ is a standard Young tableau of shape $\lambda \vdash n$ and $Q$ is a standard Young tableau of shape $\mu \vdash k$ such that $\lambda / \mu$ is a horizontal strip.
\end{thm}

For example, let $n = 4$ and $k = 2$. There are $4!/2! = 12$ injections from $[2]$ to $[4]$:
$$(1,2),(1,3),(1,4),(2,1),(2,3),(2,4),(3,1),(3,2),(3,4),(4,1),(4,2),(4,3).$$
Their respective unique pairs $(P,Q)$ of standard Young tableau are listed from left to right as follows:
\[
\begin{array}{lll}
\young(1234)~\young(12\times\times) \quad & \young(124,3)~\young(12\times,\times) \quad & \young(123,4)~\young(12\times,\times)\\
\\
\young(134,2)~\young(1\times\times,2) \quad & \young(134,2)~\young(12\times,\times) \quad & \young(13,24)~\young(12,\times\times)\\
\\
\young(124,3)~\young(1\times\times,2) \quad & \young(14,2,3)~\young(1\times,2,\times) \quad & \young(12,34)~\young(12,\times\times)\\
\\
\young(123,4)~\young(12\times,\times) \quad & \young(13,2,4)~\young(1\times,2,\times) \quad & \young(12,3,4)~\young(1\times,2,\times).\\
\end{array}
\]
The theorem above gives a combinatorial proof of a natural generalization of equation~(\ref{eq:rs}).
\begin{cor}\label{cor:count}
	The number of injections from $[k]$ to $[n]$ can be counted as follows: 
	$$|S_{k,n}| = \sum_{\mu,\lambda} f^{\mu} f^{\lambda}$$
	where the sum runs over pairs $\mu \vdash k,\lambda \vdash n$ such that $\lambda / \mu$ is a horizontal strip.
\end{cor} 
In the next section, we present the corroborating representation theory for the count above. 

\section{Representation Theory and the Injection Scheme}

Throughout this work, we assume a general understanding of group representation theory and the theory of association schemes. We refer the reader to~\cite{Diaconis88} and~\cite{Godsil} for more detailed discussions. 

\subsection{Association Schemes and Finite Gelfand Pairs}
\label{sec:assoc}

Let $X$ be a finite set of cardinality $v$. An \emph{association scheme} is a set of $d+1$ binary $v \times v$ matrices $\mathcal{A} = \{A_0,A_1,\cdots,A_d\}$ over a set $X$ that satisfy the following axioms:
\begin{enumerate}
	\item $A_i \in \mathcal{A} \Rightarrow A_i^\top \in \mathcal{A}$ for all $0 \leq i \leq m$,
	\item $A_0 = I$ where $I$ is the identity matrix,
	\item $\sum_{i=0}^d A_i = J$ where $J$ is the all-ones matrix, and 
	\item $A_iA_j = A_jA_i = \sum_{k = 1}^{d} p_{i,j}(k)A_k$ and $p_{i,j}(k) \in \mathbb{Z}$ for all $0 \leq i,j \leq d$.
\end{enumerate}
Moreover, if $A_i = A_i^\top$ for all $0 \leq i \leq m$, then we say the association scheme is \emph{symmetric}.

The matrices $A_1,\cdots,A_d$ are referred to as \emph{associates}, and the constants $p_{i,j}(k)$ are the so-called \emph{intersection numbers} of the association scheme. Let $v_i$ denote the \emph{valency} (rowsum) of the associate $A_i$, and let $m_i$ denote the \emph{multiplicity} (dimension) of the $i$th eigenspace for all $0 \leq i \leq m$.
The matrix algebra generated by the identity matrix and its associates is the association scheme's \emph{Bose-Mesner algebra}, and these matrices form a basis for this algebra. 
The \emph{character table} of an association scheme is a $(d+1) \times (d+1)$ matrix $P$ whose rows are indexed by eigenspaces, columns indexed by matrices of $\mathcal{A}$, and defined such that $P_{i,j}$ is the eigenvalue for the $i$th eigenspace of $A_j$. It turns out that $P$ is invertible, and so the \emph{dual character table} $Q$ of the association scheme is defined to be $Q = vP^{-1}$. The dual character table of the injection scheme will be central for obtaining linear programming bounds on injection codes.

In general, one can determine the intersection numbers of an association scheme from its character table by appealing to the following well-known relation; see, for instance, Chris Godsil's notes \cite{Godsil}.
\begin{prop}
The intersection numbers satisfy
\[p_{ij}(k) = \frac{1}{vv_k} \sum_{h=0}^d m_h P_{h,i} P_{h,j} P_{h,k}.\]
\end{prop}
Conversely, it is possible to find the eigenvalues from the intersection numbers by computing eigenvalues of `random' linear combinations of intersection matrices $B_k=[p_{ij}(k)]$, which can be shown \cite[p. 13]{Delsarte} to furnish another basis for the Bose-Mesner algebra.

For association schemes that arise from groups, the entries of $P$ can be determined via group representation theory, i.e., in terms of the spherical functions of a finite \emph{Gelfand pair}~\cite{BannaiI84}.  
\begin{thm}\cite{MacDonald95}\label{thm:gelfandPair}
	 Let $K \leq G$ be a group.  Then the following are equivalent.
	\begin{enumerate}
		\item $(G,K)$ is a Gelfand Pair;
		\item The induced representation $1 \uparrow_K^G  \cong \bigoplus_{i=1}^d V_i$ (equivalently, the permutation representation of $G$ acting on $G/K$) is multiplicity-free;
		\item The double-coset algebra $\mathbb{C}[K \backslash G /K]$ is commutative.
	\end{enumerate}
	Moreover, a Gelfand pair is symmetric if $KgK = Kg^{-1}K$ for all $g \in G$.
\end{thm}
\noindent Let $(G,K)$ be a Gelfand pair, $X := G/K$, and define $\chi_i$ to be the character of $V_i$ as in the second statement of Theorem~\ref{thm:gelfandPair}, with dimension $d_i := \chi_i(1)$. The functions $\omega^1, \omega^2, \cdots,\omega^d \in \mathbb{C}[X]$ defined such that  
\[  \omega^i(g) = \frac{1}{|K|} \sum_{k \in K} \chi_i (g^{-1}k) \quad \forall g \in G\]
are called the \emph{spherical functions} and form an orthogonal basis for $\mathbb{C}[K \backslash G /K]$. We call the equation above \emph{the projection formula}, as $\omega^i$ is the projection of $\chi_i$ onto the space of $K$-invariant functions. 

The (left) $K$-orbits of $X$ partition the cosets into $(K\backslash G /K)$-double cosets, which correspond to \emph{spheres} $\Omega_0, \Omega_1, \cdots, \Omega_d \subseteq G/K$. It is helpful to think of spheres and spherical functions as the spherical analogues of conjugacy classes and irreducible characters respectively.  Indeed, the spherical functions are constant on spheres, and it can be shown that the number of distinct spherical functions equals the number of distinct irreducibles of $\mathbb{C}[X]$, equivalently, the number of spheres of $X$.  

We write $\omega^i_j$ for the value of the spherical function $\omega^i$ corresponding to the $i$th irreducible on the double coset corresponding to $\Omega_j$.
\begin{prop}\cite{BannaiI84}\label{prop:eig}
	Let $(G,K)$ be a finite Gelfand pair and let $P$ be the character table of the corresponding association scheme. Then
	\[ P_{i,j} = |\Omega_j| \omega^i_j.\]
\end{prop}
\subsection{The Injection Scheme}

In this section we recall some basic facts about the \emph{injection scheme}, a symmetric association scheme defined over the set of injections. For proofs of the following basic facts and a more detailed discussion of the injection scheme, we refer the reader to~\cite{LR}.\footnote{In~\cite{LR}, the scheme is called the $k$-partial permutation association scheme, as one may interpret injections as partial permutations.} 

The product $G := S_k \times S_n$ of two symmetric groups acts on an injection $\sigma \colon[k]\rightarrow[n]$ as
$(\pi,\rho)\colon \sigma \mapsto \rho *\sigma*\pi^{-1}$, where $(\pi,\rho)\in G$ and $*$ denotes the composition of functions. The stabilizer of the identity injective function with respect to this action is the group $K_{k,n}$, i.e., the cosets $G_{k,n}/K_{k,n}$ are in one-to-one correspondence with injective functions. This action gives a permutation representation $1 \uparrow_K^G$ that is multiplicity-free, i.e., $(G_{k,n},K_{k,n})$ is a symmetric Gelfand pair. By the Littlewood-Richardson rule, we have
\begin{align}
	1 \!\uparrow_{K_{k,n}}^{G_{k,n}}~~ \cong\!\!\!\!\!\!\! \bigoplus_{ \substack{ \mu \vdash k,\lambda \vdash n \\ \lambda/\mu \text{ is a horiz.~strip} }} \!\!\!\!\!\!\! \mu \otimes \lambda.
\end{align}
The orbitals of $G$ acting diagonally on $G/K \times G/K$ are in one-to-one correspondence with double cosets $K \backslash G / K$. If we think of injections graphically as maximum matchings of the complete bipartite graph $K_{k,n}$, then we observe that the double cosets and orbitals are in one-to-one correspondence with graph isomorphism classes that arise from the multiunion of any injection with the identity injection, i.e., a disjoint union of even paths and even cycles. In light of this, we use the notation $(\lambda | \rho)$ to denote the index of the orbital or double coset corresponding to the isomorphism class containing a cycle of length $2\lambda_i$ for all $1 \leq i \leq \ell(\lambda)$, and a path of length $2\rho_i$ for all $1 \leq i \leq \ell(\rho)$. Note that an isolated node is a path of length zero.

Let $C_{(\lambda | \rho)}$ denote the sphere corresponding to the cycle-path type $(\lambda|\rho)$.
For example, we have $(1,2,3,4) \in C_{(1^4 | 0^4)}$, $(2,1,3,5) \in C_{(2,1|0^3,1)}$, and $(5,6,7,8) \in C_{( \varnothing | 1^4)}$. The following result gives a simple count for the sizes of the spheres, analogous to the well-known formula for determining the size of a conjugacy class in $S_n$. 
\begin{prop}\label{prop:spheresizes}
	\cite{LR} For any cycle-path type $(\lambda | \rho)$, the size of the $(\lambda|\rho)$-sphere is
	\[ |C_{(\lambda | \rho)}| = \frac{k!(n-k)!}{\prod_{i=0}^k i^{\ell_i} \ell_i!r_i! } \]
	where $\lambda=(0^{\ell_0},1^{\ell_1},\cdots,k^{\ell_k})$f, $\rho=(0^{r_0}, 1^{r_1},\cdots,k^{r_k})$, and $\ell(\rho)=r_1+\cdots+r_k$.
\end{prop} 
Note that the orbitals can be represented as a set $\mathcal{A}_{k,n} := \{A_{(\lambda|\rho)}\}$ of symmetric matrices that sum to the all-ones matrix. In particular, we have
\[
A_{(\lambda|\rho)} = 
\begin{cases}
1 \quad &\text{if } i \cup j \cong (\lambda|\rho), \\
0 \quad &\text{otherwise},
\end{cases}
\]  
for all injections $i,j$ and cycle-path types $(\lambda|\rho)$. We call $\mathcal{A}_{k,n}$ \emph{the injection scheme}, or more precisely, \emph{the $(k,n)$-injection scheme}. Note that the valencies $v_{(\lambda | \rho)}$ are just the sizes of the spheres, and the multiplicity $m_{(\lambda | \rho)}$ is the dimension of the irreducible representation corresponding to $(\lambda | \rho)$. This correspondence between cycle-path types and irreducible representations can be described as follows.

Recall that irreducible representation that appear in the permutation representation of $G$ on injections is of the form $\alpha \otimes \beta$ where $\beta / \alpha$ is a horizontal strip of size $n-k$. Consider a tableau of $\beta$ such that the cells of $\beta / \alpha$ are marked.
Every column of $\alpha$ in $\beta$ with a marked cell below it corresponds to a part in $\rho$ whereas an unmarked column corresponds to a part in $\lambda$. 
For instance, taking $\alpha = (2,1)$ and $n = 7$, we have the following cycle-path types for varying $\alpha \otimes \beta$:
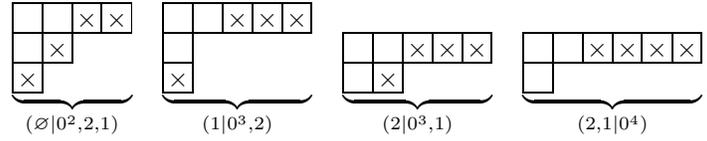
\[\underbrace{\young(~~\times\times,~\times,\times)}_{(\varnothing | 0^2,2,1)} \quad 
\underbrace{\young(~~\times\times\times,~,\times)}_{(1|0^3,2)} \quad 
\underbrace{\young(~~\times\times\times,~\times)}_{(2|0^3,1)} \quad  
\underbrace{\young(~~\times\times\times\times,~)}_{(2,1|0^4)}. \]
Note that marked singleton columns correspond to paths of length zero (i.e., isolated nodes).

We are now in a position to give a formula for $\omega^{(\lambda|\rho)}_{(\mu|\nu)}$ for all cycle-path types. Let $\alpha \otimes \beta$ be the irreducible representation of $G$ represented by the cycle-path type $(\lambda|\rho)$. We may pick a double coset representative $(\tau,\sigma) \in G$ of $(\mu|\nu)$ such that $\tau = ()$, as the one-sided action of $S_n$ on injections is transitive. We have 
	\begin{align*}
	\omega^{(\lambda|\rho)}_{(\mu|\nu)} = \omega^{(\lambda|\rho)}(((),\sigma)) &= \frac{1}{|K|}\sum_{k\in K} \chi_{\alpha \otimes \beta}(((),\sigma)^{-1}k)\\
	&= \frac{1}{|K|}\sum_{(k_1,k_2) \in K} \chi_{\alpha \otimes \beta}((k_1,\sigma^{-1}k_1k_2))\\
	&= \frac{1}{k!(n-k)!}\sum_{k_1 \in S_k} \chi_{\alpha}(k_1) \sum_{k_2 \in S_{n-k}} \chi_{\beta}(\sigma^{-1}k_1k_2).
	\end{align*}
Note that the entries of the character table of any symmetric association scheme are algebraic integers, and the characters of the symmetric group are integers; therefore, the projection formula above shows that the entries of the character table of the injection scheme are integers. 
As an aside, this gives a much simpler proof of the integrality of the spectra of so-called \emph{$(n,k,r)$-arrangement graphs}, which live in the Bose-Mesner algebra of $\mathcal{A}_{k,n}$ (see~\cite{Chen} for more details).

Although the projection formula gives an explicit way of computing the character table of $\mathcal{A}_{k,n}$, it is difficult to work with from both a computational and analytical point of view. It becomes prohibitively difficult to compute the character table of $\mathcal{A}_{k,n}$ using this formula for even modest values of $k,n$, and it seems difficult to derive good expressions for the characters of $\mathcal{A}_{k,n}$ using this formula. Indeed, we are unaware of any result that leverages the projection formula for spherical functions to derive tractable expressions for the character tables of association schemes associated with Gelfand pairs.  

\section{A Canonical Basis for the Injection Scheme}

Let $(\rho_1,V_1)$ and $(\rho_2,V_2)$ be two representations of a group $H$, and let $\phi : V_1 \rightarrow V_2$ be a linear transformation. We say that $\phi$ \emph{intertwines} $\rho_1$ and $\rho_2$ if $ \phi \rho_1(h) =  \rho_2(h) \phi$ for all $h \in H$. 

\begin{lemma}[Schur's Lemma]
	If $(\rho_1,V_1)$ and $(\rho_2,V_2)$ are irreducible representations of $H$ and $\phi$ is an intertwining map for $\rho_1$ and $\rho_2$, then either $\phi$ is the zero map or it is an isomorphism.
\end{lemma}

Let $\mathbb{C}[S_{k,n}]$ be space of all complex-valued functions defined over injections $S_{k,n}$. Let $\{e_i\}$ defined such that $e_i(j) = \delta_{i,j}$ for all $i,j \in S_{k,n}$ be the standard basis for this space.
For any $\lambda \vdash n$, let $M^\lambda$ be permutation representation of $S_n$ acting on the set of all $\lambda$-tabloids. Let $\{e_{\{t\}}\}$ defined such that $e_{\{t\}}(\{s\}) = \delta_{\{t\},\{s\}}$ for any two $\lambda$-tabloids $\{t\},\{s\}$ be the standard basis for this space. The product $M^\mu \otimes M^\lambda$ is a $G_{k,n}$-representation with basis $\{e_{\{s\}} \otimes e_{\{t\}}\}$ where $\{s\}$,$\{t\}$ range over all $\mu$-tabloids and $\lambda$-tabloids respectively.

Let $\{s\}$ be a $\mu$-tabloid and $\{t\}$ be a $\lambda$-tabloid such that $\mu \vdash k$ and $\lambda \vdash n$. We say that $\{s\},\{t\}$ \emph{covers} an injection $\sigma \in S_{k,n}$ if $\text{row}_{\{s\}}(i) = \text{row}_{\{t\}}(\sigma(i))$ for all $1 \leq i \leq k$. For example, the injections $(1,2,3,4,5)$ in red and $(2,3,6,5,4)$ in thick blue are covered by the tabloids in normal and bold lettering below, whereas the injection $(4,3,6,5,2)$ in dashed blue is not:
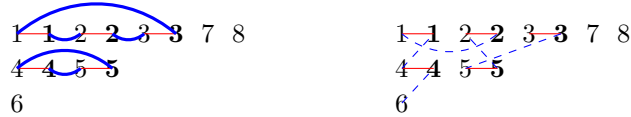
\begin{center}
\usetikzlibrary{arrows, fit, matrix, positioning, shapes, backgrounds,
}
\begin{tikzpicture}
\matrix (m) [
matrix of math nodes, 
nodes in empty cells,
minimum width=width("8"),
] {
	1 & \mathbf{1} & 2 & \mathbf{2} & 3 & \mathbf{3} & 7 & 8 \\
	4 & \mathbf{4} & 5 & \mathbf{5} \\
	6  \\
};

\draw[red] (m-1-1.center) -- (m-1-2.center);
\draw[red]  (m-1-3.center) --  (m-1-4.center);
\draw[red]  (m-1-5.center) -- (m-1-6.center);
\draw[blue,very thick] (m-1-2.center) to [bend right = 40]  (m-1-3.center);
\draw[blue,very thick] (m-1-4.center) to [bend right = 40]  (m-1-5.center);
\draw[blue,very thick] (m-1-1.center) to [bend left = 40]  (m-1-6.center);

\draw[red]  (m-2-1.center) --  (m-2-2.center);
\draw[red]  (m-2-3.center) --  (m-2-4.center);

\draw[blue,very thick] (m-2-1.center) to [bend left = 40]  (m-2-4.center);
\draw[blue,very thick] (m-2-2.center) to [bend right = 40]  (m-2-3.center);
       
\end{tikzpicture}
\quad\quad\quad\quad\begin{tikzpicture}
\matrix (m) [
matrix of math nodes, 
nodes in empty cells,
minimum width=width("8"),
] {
	1 & \mathbf{1} & 2 & \mathbf{2} & 3 & \mathbf{3} & 7 & 8 \\
	4 & \mathbf{4} & 5 & \mathbf{5} \\
	6  \\
};

\draw[red] (m-1-1.center) -- (m-1-2.center);
\draw[red]  (m-1-3.center) --  (m-1-4.center);
\draw[red]  (m-1-5.center) -- (m-1-6.center);
\draw[blue,dashed] (m-2-4.center) --  (m-1-3.center);
\draw[blue,dashed] (m-1-1.center) to [bend right = 40]  (m-1-4.center);
\draw[blue,dashed](m-3-1.center) -- (m-2-2.center);

\draw[red]  (m-2-1.center) --  (m-2-2.center);
\draw[red]  (m-2-3.center) --  (m-2-4.center);f

\draw[blue,dashed] (m-2-1.center) -- (m-1-2.center);
\draw[blue,dashed] (m-1-6.center) --   (m-2-3.center);

\end{tikzpicture}.
\end{center}
Let $1_{\{s\},\{t\}} \in \mathbb{C}[S_{k,n}]$ be the characteristic function of the set of injections covered by $\{s\},\{t\}$. For any $\mu \vdash k,\lambda \vdash n$ such that $\lambda / \mu$ is a horizontal strip, let $\phi_{\mu,\lambda} : M^\mu \otimes M^\lambda \rightarrow \mathbb{C}[S_{k,n}]$ be the map defined such that
\[\phi_{\mu,\lambda}(e_{\{s\}} \otimes e_{\{t\}}) = 1_{\{s\},\{t\}} \quad \text{ for all }\{s\},\{t\},  \]
then extending linearly. An injection $\sigma$ is covered by $\{s\}, \{t\}$ if and only if $(\tau, \pi) \sigma$ is covered by $(\{\tau s\}, \{\pi t\})$ for all $(\tau,\pi) \in G_{k,n}$. This implies that 
	\[ \phi_{\mu,\lambda}( \tau e_{\{s\}} \otimes \pi e_{\{t\}} ) = (\tau,\pi)\phi_{\mu,\lambda}( e_{\{s\}} \otimes e_{\{t\}} )\quad \text{for all } (\tau,\pi) \in G_{k,n}, \]
	i.e., the linear map $\phi_{\mu,\lambda}$ intertwines $M^\mu \otimes M^\lambda$ and $\mathbb{C}[S_{k,n}]$.

%
It is well-known that the $\lambda$-isotypic component of $M^\lambda$ has multiplicity 1, and so the $(\mu \otimes \lambda)$-isotypic component of $M^\mu \otimes M^\lambda$ has multiplicity 1.  
Let $(\rho_{\mu,\lambda},V_\mu \otimes V_\lambda)$ be this $G_{k,n}$-irreducible. A basis for $\rho_{\mu,\lambda}$ can be obtained by tensoring all pairs of standard $\mu$-polytabloids and standard $\lambda$-polytabloids. For each standard Young tableau $t$, let $e_t$ denote the corresponding standard polytabloid.

We say that an injection $\sigma$ is \emph{aligned} with respect to $\{s\},\{t\}$ if $\text{row}_{\{s\}}(i) = \text{row}_{\{t\}}(\sigma(i))$ and $\text{col}_{s}(i) = \text{col}_{t}(\sigma(i))$ for all $1 \leq i \leq k$. In the example above, the blue dashed injection $(2,3,6,5,4)$ is not aligned with the tabloids above, but the red injection $(1,2,3,4,5)$ is.
\begin{lemma}\label{lem:nz}
	For each irreducible representation $V_\mu \otimes V_\lambda$ of the induced representation $1\!\!\uparrow_{K_{k,n}}^{G_{k,n}}$, there exists a $v \in V_\mu \otimes V_\lambda$ such that $\phi_{\mu,\lambda}(v) \neq 0$.
\end{lemma}
\begin{proof}
	Let $e$ be the identity injection. Consider the pair of standard Young tableaux $s,t$ of shape $\mu$ and $\lambda$ respectively obtained by inserting the numbers $1,2,\cdots,k$ into the rows of $s$ from left to right, top to bottom, then taking $t$ to be the standard Young tableau obtained from $s$ by adding a horizontal strip and labeling the cells $k+1,k+2, \cdots, n$ from left to right. For example, if $\mu = (3,2,1)$ and $\lambda = (4,3,2)$, then $s$ and $t$ are
	\[\young(123,45,6) \quad \quad\quad\quad\quad
	\young(1239,458,67). \]
	Note that $e$ is aligned with respect to $\{s\},\{t\}$. Let $C_s,C_t$ denote the column-stabilizers of $s$ and $t$ respectively. It is clear that 
	\[e_s \otimes e_t = \sum_{\pi \in C_s,\pi' \in C_t} \text{sgn}(\pi)~\text{sgn}(\pi')~ e_{\{\pi s\}} \otimes e_{\{\pi' t\}}.\]
	Let $v = e_s \otimes e_t$ and $f = \phi_{\mu,\lambda}(v)$. We have
	\[f(e) = \sum_{\pi \in C_s,\pi' \in C_t} \text{sgn}(\pi)~\text{sgn}(\pi')~ 1_{\{\pi s\},\{\pi' t\}}(e). \]
	If $\pi \in C_s$ sends $i$ to $j$ such that $1 \leq i,j \leq k$, then $\pi' \in C_t$ must also send $i$ to $j$, otherwise $\{\pi s\},\{\pi' t\}$ does not cover $e$. On the other hand, if $\pi' \in C_t$ sends $i$ to $j$ such that $1 \leq i \leq k$ and $k+1 \leq j \leq n$, then $(\{\pi s\},\{\pi' t\})$ does not cover $e$ for all $\pi \in C_s$, which implies that the cells of the horizontal strip $\lambda/\mu$ are fixed points of every $\pi' \in C_t$ such that $\{\pi s\},\{\pi' t\}$ covers $e$. The foregoing implies that $\text{sgn}(\pi)~\text{sgn}(\pi') = 1$ if and only if $\{\pi s\},\{\pi' t\}$ covers $e$. In particular, we have
	\[f(e) = \sum_{\pi \in C_s,\pi' \in C_t} \text{sgn}(\pi)~\text{sgn}(\pi')~ 1_{\{\pi s\},\{\pi' t\}}(e) = |C_s|, \]
	thus $f = \phi_{\mu,\lambda}(v)  \neq 0$, as desired.
\end{proof}
\noindent Now let $f_{s,t} := \phi_{\mu, \lambda}$ where $s,t$ are standard Young tableaux of shape $\mu \vdash k$ and $\lambda \vdash n$ such that $\lambda / \mu$ is a horizontal strip.
Let $\mathcal{F} := \{f_{s,t}\}$ where $s$ and $t$ range over all such standard Young tableaux. By Lemma~\ref{lem:nz}, $\phi_{\mu,\lambda}$ is not the zero map, so by Schur's Lemma, we have that $\phi_{\mu,\lambda}$ is an isomorphism. 
	The elements of $\mathcal{F}$ are pairwise linearly independent, hence Corollary~\ref{cor:count} implies that $\mathcal{F}$ is a basis. Moreover, we have the property that basis functions in different isotypic components are orthogonal, thus we arrive at the following theorem.
\begin{thm}
	The set $\mathcal{F}$ is a basis for $\mathbb{C}[S_{k,n}]$ such that $\langle f_{q,r},f_{s,t}\rangle = 0$ for all $f_{q,r} \in V_{\mu \otimes \lambda}$ and $f_{s,t} \in V_{\mu' \otimes \lambda'}$ such that $\lambda / \mu \neq \lambda' / \mu'$.
\end{thm}
\noindent It would be interesting to refine the result above to a \emph{Fourier basis} for $\mathbb{C}[S_{k,n}]$, that is, further require that basis functions in the same isotypic component are orthogonal. Note that Young's orthogonal form furnishes such a basis for the case  $k=n$ (see~\cite{Diaconis88}).
\begin{thm}[Frobenius Reciprocity]\label{thm:frobenius}
	Let $\rho$ be an irreducible representation of a group $H$ and let $K$ be a subgroup of $H$. The multiplicity of the $\rho$-isotypic component of $1\!\uparrow^H_K$ is the dimension of the subspace of $K$-invariant functions of the $\rho$-isotypic component.
\end{thm}
\noindent Let $Q_{k,n}$ denote the projection onto the space of $K_{k,n}$-invariant functions. For any $\mu \vdash k$, define $\mu! := \mu_1!\mu_2!\cdots \mu_{\ell(\mu)}!$.
\begin{lemma}
	Let $s,t$ be standard Young tableaux of shape $\mu \vdash k$ and $\lambda \vdash n$ such that $\lambda / \mu$ is a horizontal strip. If $Q_{k,n}f_{s,t} \neq 0$, then $\frac{1}{(\mu^\top)!}Q_{k,n}f_{s,t}$ is the $(\mu \otimes \lambda)$-spherical function. 
\end{lemma}
\begin{proof}
	By construction, $f_{s,t} \in \mathbb{C}[S_{k,n}]$ lives in the irreducible $W \leq \mathbb{C}[S_{k,n}]$ that is isomorphic to $V_{\mu} \otimes V_{\lambda}$. Because $Q_{k,n}$ sends $W$ to $W$, we have that $Q_{k,n}f_{s,t} \in W$ is a $K_{k,n}$-invariant function. By Frobenius Reciprocity, the space of $K_{k,n}$-invariant functions of $W$ has dimension 1; therefore, if $Q_{k,n}f_{s,t} \neq 0$, then it is the $(\mu \otimes \lambda)$-spherical function up to scaling. To ensure that the $(\mu,\lambda)$-spherical function is 1 on the $K_{k,n} \backslash () /K_{k,n}$ double coset, we normalize by $|C_s| = (\mu^\top)!$.
\end{proof}
\noindent We are now ready to give a proof of our formula for the spherical functions of $(G_{k,n},K_{k,n})$. Let $s,t$ be the pair of standard Young tableaux as defined in the proof of Lemma~\ref{lem:nz}.
\begin{thm}\label{thm:main}
	Let $\omega^{\mu \otimes \lambda}$ be the $(\mu \otimes \lambda)$-spherical function of the Gelfand pair $(G_{k,n},K_{k,n})$. Then
	\[ \omega^{\mu \otimes \lambda}_{(\gamma|\rho)} =  \frac{1}{|C_{(\gamma|\rho)}|} \sum_{\sigma \in C_{(\gamma|\rho)}}\sum_{\pi \in C_t} \emph{sgn}(\pi)~1_{\{s\}, \{\pi t\}}(\sigma), \text{ and } \]
	\[P_{\mu \otimes \lambda, (\gamma|\rho)} = \sum_{\sigma \in C_{(\gamma|\rho)}}\sum_{\pi \in C_t} \emph{sgn}(\pi)~1_{\{s\}, \{\pi t\}}(\sigma)\]
	for all cycle-path types $(\gamma|\rho)$. Moreover, .
\end{thm}
\begin{proof}
	An argument similar to the proof of Lemma~\ref{lem:nz} shows that $Q_{k,n}f_{s,t} \neq 0$, hence $Q_{k,n}f_{s,t} = \omega^{(\mu \otimes \lambda)}$ by the lemma above. In particular, we have
	\[ \omega^{\mu \otimes \lambda}_{(\gamma|\rho)} =  \frac{1}{(\mu^\top)!|C_{(\gamma|\rho)}|} \sum_{\pi \in C_s, \pi' \in C_t} \text{sgn}(\pi)~\text{sgn}(\pi') |\{ \sigma \in C_{(\gamma|\rho)} : \{\pi s\}, \{\pi't\}\text{ covers } \sigma \}|. \]
	But note that $C_s \leq C_t$, which gives us
	\[ \omega^{\mu \otimes \lambda}_{(\gamma|\rho)} =  \frac{1}{(\mu^\top)!|C_{(\gamma|\rho)}|} \sum_{\pi \in C_s, \tau \pi  \in  C_t\pi } \text{sgn}(\tau) |\{ \sigma \in C_{(\gamma|\rho)} : \{\pi s\}, \{\tau \pi t\}\text{ covers } \sigma \}|. \]
	Since  $\{\pi s\}, \{\tau \pi t\}\text{ covers } \sigma$ if and only if  $\{s\}, \{\tau t\}\text{ covers } \sigma$, we may rewrite the above as
	\[ \omega^{\mu \otimes \lambda}_{(\gamma|\rho)} =  \frac{1}{|C_{(\gamma|\rho)}|} \sum_{\pi \in C_t} \text{sgn}(\pi) |\{ \sigma \in C_{(\gamma|\rho)} : \{s\}, \{\pi t\}\text{ covers } \sigma \}|. \]
	Rearranging completes the proof of the first part, and Proposition~\ref{prop:eig} proves the second part. 
\end{proof}

\noindent To give a quick demonstration of the formula's efficacy, let $\overline{\mu} := (\mu_1 + n -k,\mu_2,\cdots,\mu_{\ell(\mu)}) \vdash n$ for any $\mu \vdash k$, and consider the spherical function $\omega^{\mu \otimes \overline{\mu}}$. For any $(\gamma|\rho)$ such that $\rho$ has more than $\mu_1$ non-trivial paths, for each $\sigma \in C_{(\gamma|\rho)}$, there exist two cells $c_\sigma,c'_\sigma$ in the same column of $t$ that are not in the image of $\sigma$. Each of these involutions $(c_\sigma,c'_\sigma)$ are sign-reversing, showing that  $\omega^{\mu \otimes \overline{\mu}}_{(\gamma|\rho)} = 0$, which is hardly transparent from the projection formula.   
Indeed, the $(\mu \otimes \overline{\mu})$-spherical functions play a crucial role in~\cite{Ambainis,LR}, and our formula may allow one to improve the results of~\cite{LR}.


\section{Code bounds}
\label{sec:codes}

\subsection{Delsarte's linear programming bound}

Consider an association scheme $\mathcal{A}=\{A_0,A_1,\dots,A_d\}$ over $X$, as defined in Section~\ref{sec:assoc}. 

For a subset $Y \subseteq X$, its characteristic function $\bm{\phi}=\bm{\phi}_Y \in \{0,1\}^X$ is defined in the usual way as
\[
\bm{\phi}(x) = 
\begin{cases}
1 & \text{if}~ x \in Y;\\
0 & \text{otherwise}.
\end{cases}
\]
Assuming $Y \neq \emptyset$, its \emph{inner distribution vector} $\mathbf{a}=\mathbf{a}_Y=(a_0,a_1,\dots,a_d)$ has entries
\[
a_i = \frac{\bm{\phi}^\top A_i \bm{\phi}}{\bm{\phi}^\top  \bm{\phi}},
\]
representing the relative frequencies of $i$th associates among pairs of elements of $Y$.

The following observation is simple, yet has profound consequences.

\begin{thm}[Delsarte, \cite{Delsarte}]
\label{Delsarte-LP}
For $\emptyset \neq Y \subseteq X$, its inner distribution vector $\mathbf{a}$ satisfies
\[\mathbf{a} Q \ge \mathbf{0},\]
where $Q$ is the dual eigenmatrix.
\end{thm}

A linear programming bound for association schemes is carried by Theorem~\ref{Delsarte-LP} in the following sense.  For a set $D$ of associate indices, a $D$-\emph{code} is a subset $Y \subseteq X$ such that any two distinct elements of $Y$ are $i$th associates for some $i \in D$.  In other words, in terms of the distribution vector, $Y$ is a $D$-code if $a_i=0$ for $i \not\in D$.  Moreover, distribution vectors are normalized so that $a_0=1$. Consider the LP
\begin{align}
\label{the-lp}
\text{maximize} \hspace{5mm} & \sum_i a_i\\
\nonumber
\text{subject to} \hspace{5mm} & \mathbf{a} Q \ge \mathbf{0},~a_0 = 1 \text{ and } a_i =0 ~\text{for}~ i \in D.
\end{align}
It follows from Theorem~\ref{Delsarte-LP} and the above remarks that the cardinality of a $D$-code is upper bounded by the solution of \eqref{the-lp}.

Using \eqref{the-lp}, Tarnanen \cite{T} computed LP bounds on permutation codes for $n \le 10$ and various allowed distance sets.  This was extended by Bogaerts \cite{Bogaerts} to $n \le 14$.  Here, the character table of $S_n$ takes the role of $Q$, and the set $D$ is built to contain all $\lam \vdash n$ such that $n$ minus the number of ones in $\lam$ belong to the allowed set of Hamming distances.  For example, with $n=5$ and Hamming distances in $\{2,4,5\}$, we take $D$ to consist of the associate indices
\[(2,1,1,1),(2,2,1),(4,1),(3,2),(5).\]
With this in mind, we abuse notation and henceforth let $D$ simply denote our set of allowed Hamming distances.

For the injection scheme $\mathcal{A}_{k,n}$, we implemented \eqref{the-lp} for $3 \le k < n \le 15$,
with the exception of a few parameter pairs $(k,n)$ at the larger end of this triangle, which are presently out of reach.  Whereas character tables for the symmetric group, and even for $\mathcal{A}_{n-1,n}$ using \cite{Strahov}, can be computed recursively for moderately large $n$, the character table for general $k$ and $n$ is presently more challenging computationally, even with our combinatorial formula from Section 4. On the other hand, the LP itself is comparatively easy to solve (and check numerically), even for parameters near the upper end of our range.

In what follows, we let $M(n,k,D)$ denote the maximum size of a $D$-code in the injection scheme $\mathcal{A}_{k,n}$, where $D$ contains the allowed Hamming distances.  We briefly describe a few natural distance sets $D$ and discuss some related combinatorial problems.  Tables of LP bounds accompany these different categories of codes.

\subsection{Classical minimum distance codes}

For the purpose of detecting and correcting errors, the distance sets typically considered in coding theory model a minimum allowed distance; that is, one takes $D=\{d,d+1,\dots\}$ for some integer $d$. Following the notation used in \cite{CKL,Dukes,T}, we write $M(n,k,d)$ in place of $M(n,k,\{d,d+1,\dots\})$ for the maximum size of an injection code with minimum Hamming distance $d$.  Some basic observations and bounds on $M(n,k,d)$ can be found in \cite{Dukes}.  A basic recursive upper bound is as follows.

\begin{prop}[Singleton bound]
$M(n,k,d) \le n!/(n-k+d-1)! = |S_{k-d+1,n}|.$
\end{prop}

For $r>0$, let $b_r$ be the size of the (any) ball of radius $r$ in $S_{k,n}$.
In \cite{Dukes} $b_r$ is determined as
\[ \sum_{j=0}^{\lfloor r \rfloor} \binom{k}{j} \sum_{i=0}^j (-1)^i \binom{j}{i} \frac{(n-k+j-i)!}{(n-k)!}. \]
A standard argument then gives the sphere packing bound for injections.

\begin{prop}[Sphere packing bound]
$\displaystyle M(n, k, d) \leq \frac{|S_{k,n}|}{b_{(d-1)/2}}.$
\end{prop}

Additional bounds on $M(n,k,d)$ are motivated by interest in the permutation code case, both for applications to powerline communication \cite{CKL,Huc} and as an extremal problem of independent combinatorial interest.  Indeed, equality in the Singleton bound is equivalent to existence of an ordered design; see  \cite[Section VI.38]{hcd} and \cite{Dukes} for more information.

\begin{table}[ttttp]
\begin{tabular}{cccr}
$n$ & $k$ & $d$ & $M \le$ \\
\hline
7 & 6 & 4 & 199\\
8 & 6 & 3 & 1513\\
 & 7 & 4 & 1462\\
9 & 7 & 4 & 2846\\
 & 8 & 4 & 12096\\
 &  & 5 & 2417\\
10 & 7 & 3 & 27308\\
 & 8 & 4 & 26206\\
 &  & 5 & 5039\\
 & 9 & 4 & 92418\\
 &  & 5 & 19158\\
 &  & 6 & 4991\\ 
11 & 8 & 4 & 52646\\
\\
\end{tabular}
\hspace{1cm}
\begin{tabular}{cccr}
$n$ & $k$ & $d$ & $M \le$ \\
\hline
11 & 9 & 4 & 256682\\
 &  & 5 & 47073\\
 & 10 & 4 & 936332\\
 &  & 5 & 185560\\
 &  & 6 & 42068\\
12 & 8 & 3 & 602579\\
 & 9 & 4 & 584327\\
 & 10 & 4 & 2699260\\
 &  & 5 & 471981\\
 & 11 & 4 & 10241521\\
 &  & 5 & 1922527\\
 &  & 6 & 411090\\ 
 13 & 9 & 4 & 1185053\\
\\ 
\end{tabular}
\hspace{1cm}
\begin{tabular}{cccr}
$n$ & $k$ & $d$ & $M \le$ \\
\hline
13 & 12 & 4 & 123235550\\
 &  & 5 & 23347599\\
 &  & 6 & 4687470\\
 &  & 7 & 910371\\
14 & 13 & 4 & 1621775700\\
 &  & 5 & 309490273\\
 &  & 6 & 58903464\\
 &  & 7 & 10510496\\
 &  & 8 & 2117618\\
15 & 14 & 4 & 23358981663\\
 &  & 5 & 4130012797\\
 &  & 6 & 804830167\\
 &  & 7 & 138132435\\
 &  & 8 & 24260981
\end{tabular}
\caption{Upper bounds on $M(n, k, d)$ via linear programming.}
\label{mindist}
\end{table}

The case $d=k-1$ has special significance for its connection with latin squares.  Colbourn, Kl\o ve and Ling \cite{CKL} showed that the existence of $r$ mutually orthogonal latin squares of order $n$ imply a permutation code of length $n$ and minimum distance $n-1$ of size $rn$.  Here, the code permutations correspond to the $n$ level sets occurring among each of the $r$ squares.  With this same construction, it is easy to see that the existence of $r$ mutually orthogonal $k \times n$ latin rectangles implies $M(n,k,k-1) \ge rn$.  It follows that an upper bound on $M(n,k,k-1)$ induces an upper bound on the number of mutually orthogonal $k \times n$ latin rectangles.

\begin{ex}
There exist a set of four mutually orthogonal incomplete $6 \times 6$ latin squares missing a common $2 \times 2$ subsquare.  This implies the existence of four $4 \times 6$ latin rectangles which are mutually orthogonal in the sense that superimposing any two produces no repeated pairs.  It follows that $M(6,4,3) \ge 24$; in fact, $M(6,4,3)=27$ was shown in \cite{Dukes}.
\end{ex}

Table~\ref{mindist} presents those upper bounds on $M(n,k,d)$ we found which improve the Singleton bound and the sphere packing Bound for $k < n \le 12$, all $k \leq 10$ with $n \leq 15$, $k=n-1 < 15$, and $(n, k) \in \{ (13, 9), (14, 9) \}$. There is no entry for $(n, k) = (14, 9)$ as the LP bound is never better than one of the two trivial bounds.

While building and verifying our table of code bounds, we noticed that Bogaerts' claim in \cite{Bogaerts} of there being no improvements via \eqref{the-lp} to the Singleton bound for permutation codes ($k=n$) with $14 \le n \le 16$ is incorrect. The correct values can be found in Table \ref{mindist_perm}.

\begin{table}[htbp]
\begin{tabular}{ccr}
$n$  & $d$ & $M(n,d) \le$ \\
\hline
11 & 4 & \textit{3326400}\\
 & 6 & 158617\\
 & 7 & 36718\\
12 & 4 & \textit{39916800}\\
 & 5 & \textbf{6141046}\\
 & 6 & \textbf{1766160}\\
 & 7 & 361395\\
13 & 4 & 411555972\\
 & 5 & 75789397\\
 & 6 & \textbf{21621600}\\
 & 7 & \textbf{4163390}\\
 & 8 & 879527\\
\end{tabular}
\hspace{1cm}
\begin{tabular}{ccr}
$n$ & $d$ & $M(n,d) \le$ \\
\hline
14 & 4 & 5298680543\\
 & 5 & 918752861\\
 & 6 & 255869198\\
 & 7 & 53744475\\
 & 8 & 9901953\\
 & 9 & 2083046\\
15 & 4 & 78702624000\\
 & 5 & 12053059200\\
 & 6 & 3511921683\\
 & 7 & 773606486\\
 & 8 & 130245681\\
 & 9 & 23627561\\
\end{tabular}
\caption{LP upper bounds on permutation codes with minimum distance $d$. Bounds which are obtained in \cite{Bogaerts}, but not due the LP bound, are cursive. Bounds which are correctly the LP bound in \cite{Bogaerts} are bold.}
\label{mindist_perm}
\end{table}

\subsection{Equidistant codes and general distance sets}

An \emph{equidistant permutation array}, or EPA$(n,d)$ is a subset $\Gamma \subseteq S_n$ with the property that any two distinct elements have Hamming distance exactly $d$.  In other words, an EPA$(n,d)$ is a $\{d\}$-code in (the conjugacy scheme of) $S_n$.  The problem of determining bounds on these objects dates back to the 1970s, beginning with a question of Bolton in \cite{Bolton}.  A concise survey on equidistant permutation arrays can be found in \cite[Section VI.44.5]{hcd}.  To our knowledge, the more general problem of equidistant injection codes has not been considered.  However, Huczynska \cite{Huc} has considered equidistant families for the `constant composition' variation in which permutations are replaced by codewords having every element in $[n]$ occurring equally often.

As a particular case, distance $k/2$ for injections of even length $k$ may be especially interesting for possible connections with Hadamard matrices.

Table~\ref{equidistant} presents various upper bounds on $M(n,k,\{d\})$ for $2 \le d < k < n \le 10$, that is, upper bounds on the size of a family of  injections in $S_{k,n}$ at pairwise distance exactly $d$.

\begin{table}[ttt]
\begin{tabular}{cccrr}
$n$ & $k$ & $d$ & $M \le$ & Triv \\
\hline
5 & 3 & 2 & 5 & 6\\
 & 4 & 3 & 6 & 7\\
\hline
7 & 3 & 2 & 8 & 9\\
 & 4 & 2 & 9 & 10\\
 & 5 & 3 & 14 & 15\\
 & 6 & 3 & 13 & 14\\
 &  & 4 & 26 & 30\\
\hline
8 & 6 & 3 & 20 & 21\\
 &  & 4 & 37 & 40\\
 & 7 & 4 & 30 & 32\\
 &  & 5 & 45 & 52\\
\hline
9 & 6 & 3 & 28 & 29\\
 &  & 4 & 42 & 45\\
 & 7 & 3 & 26 & 30\\
 & 8 & 3 & 20 & 21\\
 &  & 4 & 40 & 44\\
 &  & 6 & 59 & 72\\
\hline
10 & 4 & 3 & 18 & 19\\
 & 6 & 3 & 35 & 36\\
 &  & 4 & 47 & 49\\
 & 7 & 3 & 40 & 42\\
 &  & 4 & 77 & 78\\
 &  & 5 & 83 & 87\\ \\
\end{tabular}
\hspace{1cm}
\begin{tabular}{cccrr}
$n$ & $k$ & $d$ & $M \le$ & Triv \\
\hline
10 & 8 & 3 & 30 & 33\\
 &  & 6 & 92 & 107\\
 & 9 & 3 & 22 & 23\\
 &  & 4 & 47 & 58\\
 &  & 5 & 92 & 95\\
 &  & 7 & 75 & 96\\
\hline
11 & 4 & 3 & 20 & 21\\
 & 6 & 3 & 42 & 43\\
 &  & 4 & 53 & 54\\
 & 7 & 3 & 49 & 52\\
 &  & 4 & 94 & 100\\
 &  & 5 & 87 & 93\\
 & 8 & 3 & 48 & 52\\
 &  & 6 & 142 & 143\\
 & 9 & 3 & 33 & 38\\
 &  & 6 & 117 & 119\\
 &  & 7 & 108 & 141\\
 &  & 8 & 60 & 61\\
 & 10 & 4 & 52 & 67\\
 &  & 5 & 108 & 132\\
 &  & 6 & 187 & 189\\
 &  & 8 & 93 & 123\\ \\ \\
\end{tabular}
\hspace{1cm}
\begin{tabular}{cccrr}
$n$ & $k$ & $d$ & $M \le$ & Triv \\
\hline
12 & 5 & 2 & 23 & 24\\
 & 6 & 2 & 26 & 27\\
 &  & 3 & 49 & 50\\
 & 7 & 3 & 60 & 61\\
 &  & 4 & 104 & 114\\
 &  & 5 & 92 & 96\\
 & 8 & 3 & 66 & 67\\
 &  & 6 & 146 & 156\\
 & 9 & 3 & 53 & 57\\
 &  & 4 & 115 & 116\\
 &  & 7 & 166 & 191\\
 & 10 & 3 & 38 & 43\\
 &  & 4 & 114 & 119\\
 &  & 5 & 199 & 201\\
 &  & 6 & 212 & 214\\
 &  & 7 & 166 & 168\\
 &  & 8 & 126 & 180\\
 &  & 9 & 76 & 77\\
 & 11 & 3 & 28 & 29\\
 &  & 4 & 56 & 75\\
 &  & 5 & 120 & 150\\
 &  & 6 & 393 & 394\\
 &  & 7 & 317 & 324\\
 &  & 9 & 112 & 153\\
\end{tabular}
\caption{Upper bounds on $M=M(n, k, \{ d \})$ for equidistant injection codes. The column with heading `Triv' contains a trivial upper bound given by \eqref{clique-coclique}; that is,  the floor of $|S_{k,n}|$ divided by the LP bound for $M(n,k, \{d \}^c)$.}
\label{equidistant}
\end{table}

Studying various other sets of allowed distances is natural in many situations.  Delsarte showed \cite{Delsarte} that the solution $M_{LP}(D)$ to \eqref{the-lp} satisfies a clique-coclique bound 
\begin{equation}
\label{clique-coclique}
M_{\mathrm{LP}}(D) M_{\mathrm{LP}}(D^c) \le |X|,
\end{equation}
where $D^c$ denotes the complement of $D$.  In particular, it follows from \eqref{clique-coclique} that code bounds in the minimum distance case can be obtained from applying the LP \eqref{the-lp} to the pairwise \emph{maximal} distance case.  As an example, Tarnanen \cite{T} and Dukes and Sawchuk \cite{DS} compute LP bounds for some small allowed distances, such as $D=\{2,3\}$, which hold for general $n$.

Even dropping the condition that allowed distances form an interval is not without some precedent in other contexts: set systems with intersection conditions modulo a prime, arcs in finite geometries, or the independence number of relation graphs in an association scheme.

We offer a (contrived) example for injections with a distance set which is not an interval.  Recall that in `eventown', with population $N$, there is a family $\mathscr{C}$ of clubs with the property that $|C \cap C'|$ is even for any $C,C' \in \mathscr{C}$.   Berlekamp \cite{Berlekamp} showed that the number of clubs satisfies $|\mathscr{C}| \le 2^{\lfloor N/2 \rfloor}$.

{\bf Problem}.  The $N$ citizens of eventown are electing a mayor from a selection of $n$ candidates.  Each ballot consists of a ranked list of $k$ of the candidates.  Is it possible for every two ballots to agree in an even number of places?  The answer is yes if $N \le M(n,k,D)$, where $D=\{2,4,6,\dots\}$.

\begin{table}[htbp]
\begin{tabular}{cccrr}
$n$ & $k$ & $D$ & $M \leq$ & Triv \\
\hline
5 & 3 & \{1,3\} & 10 & 12\\
 & 4 & \{1,3\} & 9 & 10\\
 &  & \{2,4\} & 12 & 13\\
 &  & \{1,2,4\} & 17 & 20\\
 &  & \{1,3,4\} & 29 & 30\\
\hline
6 & 3 & \{1,3\} & 16 & 17\\
 & 5 & \{1,4\} & 23 & 26\\
 &  & \{2,4\} & 17 & 19\\
 &  & \{2,5\} & 22 & 23\\
 &  & \{3,4\} & 20 & 21\\
 &  & \{3,5\} & 25 & 26\\
 &  & \{1,2,4\} & 27 & 28\\
 &  & \{1,3,4\} & 31 & 32\\
 &  & \{1,2,5\} & 34 & 36\\
 &  & \{1,3,5\} & 37 & 42\\
 &  & \{1,4,5\} & 56 & 60\\
 &  & \{2,3,5\} & 27 & 31\\
 &  & \{2,4,5\} & 52 & 55\\
 &  & \{1,2,3,5\} & 53 & 55\\
 &  & \{1,2,4,5\} & 68 & 72\\
\hline
7 & 3 & \{1,3\} & 23 & 26\\
 & 4 & \{1,3\} & 22 & 23\\
 &  & \{1,4\} & 24 & 25\\
 &  & \{2,3\} & 33 & 35\\
 &  & \{2,4\} & 36 & 38\\
 &  & \{1,2,4\} & 72 & 76\\
 &  & \{1,3,4\} & 84 & 93\\
 & 5 & \{1,4\} & 39 & 40\\
 &  & \{1,5\} & 19 & 20\\
 &  & \{2,4\} & 34 & 36\\
 &  & \{2,5\} & 45 & 47\\
 &  & \{3,4\} & 35 & 36\\
 &  & \{3,5\} & 38 & 40\\
 &  & \{1,2,4\} & 62 & 66\\
 &  & \{1,3,4\} & 53 & 56\\
 &  & \{1,2,5\} & 69 & 72\\
 &  & \{1,3,5\} & 69 & 74\\
 &  & \{1,4,5\} & 111 & 114\\
 &  & \{2,3,4\} & 125 & 132\\
 &  & \{2,3,5\} & 63 & 64\\
 &  & \{2,4,5\} & 116 & 120\\
 &  & \{1,2,3,5\} & 158 & 168\\
 &  & \{1,2,4,5\} & 168 & 180\\
\end{tabular}
\hspace{1cm}
\begin{tabular}{cccrr}
$n$ & $k$ & $D$ & $M \leq$ & Triv \\
\hline
 & 6 & \{1,3\} & 15 & 18\\
 &  & \{1,4\} & 37 & 46\\
 &  & \{1,5\} & 40 & 42\\
 &  & \{2,4\} & 28 & 35\\
 &  & \{2,5\} & 41 & 42\\
 &  & \{3,4\} & 34 & 36\\
 &  & \{3,5\} & 41 & 43\\
 &  & \{3,6\} & 47 & 48\\
 &  & \{4,5\} & 31 & 34\\
 &  & \{4,6\} & 64 & 70\\
 &  & \{1,2,4\} & 37 & 46\\
 &  & \{1,3,4\} & 52 & 60\\
 &  & \{1,2,5\} & 71 & 78\\
 &  & \{1,3,5\} & 58 & 60\\
 &  & \{1,4,5\} & 58 & 61\\
 &  & \{1,2,6\} & 40 & 41\\
 &  & \{1,3,6\} & 74 & 75\\
 &  & \{1,4,6\} & 110 & 120\\
 &  & \{2,3,5\} & 42 & 45\\
 &  & \{2,4,5\} & 67 & 68\\
 &  & \{2,3,6\} & 82 & 86\\
 &  & \{2,4,6\} & 83 & 86\\
 &  & \{2,5,6\} & 84 & 96\\
 &  & \{3,4,5\} & 122 & 126\\
 &  & \{3,4,6\} & 64 & 70\\
 &  & \{3,5,6\} & 108 & 136\\
 &  & \{1,2,3,5\} & 72 & 78\\
 &  & \{1,2,4,5\} & 105 & 107\\
 &  & \{1,3,4,5\} & 151 & 152\\
 &  & \{1,2,3,6\} & 146 & 162\\
 &  & \{1,2,4,6\} & 117 & 122\\
 &  & \{1,3,4,6\} & 118 & 122\\
 &  & \{1,2,5,6\} & 140 & 148\\
 &  & \{1,3,5,6\} & 141 & 180\\
 &  & \{2,3,4,6\} & 118 & 126\\
 &  & \{2,3,5,6\} & 108 & 136\\
 &  & \{2,4,5,6\} & 280 & 336\\
 &  & \{1,2,3,4,6\} & 233 & 240\\
 &  & \{1,2,3,5,6\} & 168 & 193\\
 &  & \{1,2,4,5,6\} & 360 & 387\\
 \\ \\ \\
\end{tabular}
\caption{Various upper bounds on $M=M(n,k,D)$ for distance sets $D$. 
 The column with heading `Triv' contains a trivial upper bound given by \eqref{clique-coclique}; that is, the floor of $|S_{k,n}|$ divided by the LP bound for $M(n,k, D^c)$.}
\label{variousD}
\end{table}

In Table~\ref{variousD}, we present a sample of upper bounds found on $M(n,k,D)$ for small $n,k$ and sets of distances $D \subseteq [k]$. We only list entries for which Equation \eqref{clique-coclique} is not satisfies with equality, that is
\[ M_{\mathrm{LP}}(D) M_{\mathrm{LP}}(D^c) < |X|. \]
Recently, Aljohani, Bamberg, and Cameron defined the concepts of \textit{synchronizing} and \textit{separating} for association schemes \cite{ABC}. In our language, the injection scheme is \text{non-separating} if
\[ M(n, k, D) M(n, k, D^c) = |X| \]
for one non-trivial $D$. Table~\ref{variousD} implies that the injection scheme
is non-separating for $(n, k) \in \{ (5, 3), (6, 3), (7, 3) \}$ if we limit ourselves to distance-sets and not all possible graphs.

\section{Future Work and Open Questions}

\subsection{Representation Theory}

Our main question is to what extent the representation theory of the symmetric group (i.e., the Gelfand pair $(S_n \times S_n, \text{diag}(S_n))$) carries over to the Gelfand pair $(G_{k,n},K_{k,n})$. Indeed, we believe there are stronger connections to the representation theory of the symmetric group yet to be shown. 

For example, following~\cite[I.7]{MacDonald95} and letting 
\[C' := \bigoplus_{k,n~:~k\leq n} \mathbb{C}[K_{k,n} \backslash G_{k,n} / K_{k,n}],\]
one can define a natural bilinear multiplication on $C'$ so that it is a commutative and associative graded $\mathbb{C}$-algebra. Classically, the characteristic map $\text{ch} : C \rightarrow \Lambda$ is an isometric isomorphism between the commutative and associative graded algebra $C$ generated by all irreducible characters of symmetric groups and the ring of symmetric functions $\Lambda$. 
It would be particularly interesting to find an analogous characteristic map $\text{ch}':C' \rightarrow \Lambda'$ to a suitable polynomial ring $\Lambda'$ such that its vector space $(\Lambda')^k$ of degree-$k$ polynomials has dimension equal to the number of cycle-types of $S_{k,n}$. 

Finally, we suspect there are other Lie and $q$-analogues of $(G_{k,n},K_{k,n})$ that might be worth investigating, which would likely require different techniques than the ones presented here.

\subsection{Coding Theory}

It is of interest to determine when $M(n,k,D)$ can achieve its LP upper bound.  For $D=\{d,d+1,\dots,n\}$, a few constructions
can be found in \cite{Dukes}, but essentially nothing is known for other distance sets $D$.  In another direction, one can also investigate the behavior of the LP bound itself. A look at Table \ref{mindist} suggests that for fixed $k$, the LP bound is non-trivial only for a finite number of values of $n$.

As noted earlier, our data suggests that the injection scheme is usually separating except for maybe a few exceptional cases. It would be interesting to show this formally.


\end{document}